\documentclass[12pt]{amsart}

\usepackage{amsfonts}
\usepackage{color}
\usepackage{amscd}
\usepackage{epsfig}
\usepackage{graphicx}
\usepackage{amsmath,amssymb}
\usepackage{xypic}
\usepackage{enumerate}
\usepackage{amsthm}
\usepackage{url}

\newtheorem{thm}{Theorem}[section]
\newtheorem{lem}[thm]{Lemma}

\newtheorem{prop}[thm]{Proposition}
\newtheorem{cor}[thm]{Corollary}
\newtheorem{defn}[thm]{Definition}

\newtheorem{rmk}[thm]{Remark}
\newtheorem{rmks}[thm]{Remarks}
\newtheorem{ques}[thm]{Question}

\newtheorem{note}[thm]{Note}
\newtheorem{ex}[thm]{Example}

\def\O{{\mathcal O}}

\def\P{{\mathbb P}}
\def\A{{\mathbb A}}
\def\I{{\mathcal I}}

\def\Z{{\mathbb Z}}

\def\C{{\mathbb C}}
\def\Pthree{{\mathbb P}^3}

\def\Ptwo{{\mathbb P}^2}

\def\Pic{\mathop{\rm Pic}}

\def\APic{\mathop{\rm APic}}

\def\Cl{\mathop{\rm Cl}}
\def\Picloc{\mathop{\rm Picloc}}

\def\Spec{\mathop{\rm Spec}}
\def\codim{\mathop {\rm codim}}

\def\mod{\mathop{\rm mod}}
\def\m{\mathop{\rm m}}
\def\fm{\mathfrak m}

\def\mult{\mathop{\rm mult}}

\def\deg{\mathop{\rm deg}}

\def\ra{\rightarrow}

\newcommand{\ubm}[2]{\underbrace{#1}_{#2}}
\newcommand{\cyc}[1]{\langle {#1} \rangle}

\def\rdA{\mathbf A}
\def\rdD{\mathbf D}
\def\rdE{\mathbf E}

\title{Local Picard groups}

\author{John Brevik}

\address{California State University at Long Beach, 
Department of Mathematics and Statistics, Long Beach, CA 90840}

\email{jbrevik@csulb.edu}

\author{Scott Nollet}

\address{Texas Christian University, Department of Mathematics, 
Fort Worth, TX 76129}

\email{s.nollet@tcu.edu}


\begin{document}
\bibliographystyle{plain}

\begin{abstract} 
We use our extension of the Noether-Lefschetz theorem to 
describe generators of the class groups at the local rings 
of singularities of very general hypersurfaces containing a fixed 
base locus. We give several applications, including (1) every 
subgroup of the class group of the completed local ring of a rational 
double point arises as the class group of such a singularity on a 
surface in $\Pthree_\C$ and (2) every complete local ring arising 
from a normal hypersurface singularity over $\C$ is the completion 
of a unique factorization domain of essentially finite type over $\C$. 
\end{abstract}

\maketitle

\section{Introduction}

While commutative algebra is foundational to algebraic geometry, 
there are growing applications in the other direction. For example, geometric methods have been used to obtain results 
about divisor class groups of normal domains \cite[$\S 3$]{srinivas}. 
The class groups of local rings at smooth points of an algebraic variety are trivial, 
but can exhibit interesting behavior at singular points. This suggests 
investigating class groups of local rings arising from general hypersurfaces $X \subset \P^n$ 
containing a fixed base locus $Y$ that is well-behaved enough so that $X$ is normal, yet 
ill-behaved enough so that $X$ is singular. Our extension of the 
Noether-Lefschetz theorem \cite[Theorems 1.1 and 1.7]{BN} 
leads to a natural geometric description of the generators for class groups 
of local rings:

\begin{thm}\label{gens}
Let $Y \subset \P_\C^n$ be superficial and $p \in Y$. 
Then a very general hypersurface $X$ of degree $d \gg 0$ containing $Y$ is normal and 
$\Cl \O_{X,p}$ is generated by supports of the codimension $2$
irreducible components of $Y$. 
In particular, $\dim Y < n-2 \Rightarrow \Cl \O_{X,p} = 0$.
\end{thm} 

Theorem \ref{gens} provides a powerful tool for producing local rings whose class group 
satisfies some given property, as illustrated by the applications that follow. Regarding terminology, 
a closed subscheme $Y \subset \P^n$ is {\em superficial} if 
$\dim Y \leq n-2$ and the locus of embedding dimension $n$ points has dimension $\leq n-3$; 
equivalently, $Y$ lies properly on a normal hypersurface. The term {\em very general} 
refers to the complement of a countable union of proper subvarieties in the parameter space 
$\P\left(H^0\left(\P^3,\I_Y (d)\right)\right)$ and in the theorem we may take $d \geq l+2$, 
where $\I_Y (l)$ is generated by global sections. The natural map 
$\Cl \O_{X,p} \rightarrow \Cl \widehat \O_{X,p}$ being injective, we may identify $\Cl \O_{X,p}$ 
with the subgroup of $\Cl \widehat \O_{X,p}$ generated by the classes described in the theorem. 

In studying the $K$-theory of rational double points, Srinivas proved that the corresponding local 
rings are completions of UFDs \cite{srinivas1}. More generally Srinivas and Parameswaran used singular 
homology, monodromy and Picard-Lefschetz theory to prove that the local ring of {\em any} isolated local 
complete intersection singularity over $\C$ is the completion of a UFD, leading them 
to ask whether every normal complete Gorenstein local ring with coefficient field $\C$ is the completion 
of a UFD which is essentially of finite type over $\C$ \cite[Question 1]{parasrini}. 
Note that the Gorenstein condition is indispensable for such rings, since the dualizing module exists 
and corresponds to an element of the class group \cite{murthy}. 
The answer is positive in dimension $\geq 4$ by the Grothendieck-Lefschetz theorem, so the main interest is 
in dimensions two and three. 
Our contribution handles arbitrary hypersurface singularities:
\begin{thm}\label{ufd} 
Let $A=\C[[x_1, \dots x_n]]/f$, where $f$ is a polynomial defining a 
variety $V$ which is normal at the origin. 
Then there exists an algebraic hypersurface $X \subset \P^n_\C$ 
and a point $p \in X$ such that $R=\O_{X,p}$ is a UFD and $\widehat R \cong A$. 
\end{thm}

Our proof follows quickly from Theorem \ref{gens} and a 
criterion of Ruiz for recognizing isomorphic quotients 
of power series rings \cite[V, Lemma 2.2]{ruiz}. 
We conjecture the analogous result for arbitrary normal local complete intersection singularities. 

In an interesting survey article, Srinivas asks for a normal local ring $B$, which subgroups 
$\Cl A \hookrightarrow \Cl B$ arise among local rings $A$ for which 
$\widehat A = B$ \cite[Question 3.1]{srinivas}. 
Mohan Kumar proved that for almost all $\rdA_n$ and $\rdE_n$ type singularities 
on a {\em rational} surface over $\C$, the analytic isomorphism class determines the algebraic 
isomorphism class; the exceptions are $\rdA_7, \rdA_8$ and $\rdE_8$, for which there 
are two possibilities each \cite{MK}. 
In particular, the possibility for $\Cl (A) \hookrightarrow \Cl (\widehat A)$ is unique. 
Our result stands in stark contrast to his: once the rationality assumption is dropped, {\em every} possible 
subgroup arises as the class group of a local ring of a point on a surface $S \subset \Pthree$:

\begin{thm}\label{applic2}
Fix $T \in \{\rdA_n, \rdD_n, \rdE_6, \rdE_7, \rdE_8\}$ and a subgroup $H$ of the class group of the completed 
local ring for a singularity of type $T$. 
Then there exists an algebraic surface $S \subset \Pthree_\C$ and a rational double point $p \in S$ of type $T$ such that 
$\Cl \O_{S,p} \cong H$.
\end{thm}

Our proof applies Theorem \ref{gens} to a suitable base locus $Y$, when we can use 
power series techniques to compute the class group in the well-known class groups of the completed local 
rings \cite{artin,lipman}. For $\rdA_n$ singularities we used a 
Cohen-Macaulay multiplicity structure on a smooth curve for $Y$, but for $\rdD_n$ singularities 
we found it necessary to use a line with an embedded point at $p$. 
In the case of a $\rdD_n$ singularity given locally by $x^2 + y^2 z + z^{n-1} = 0$, 
the completed local ring has class group $\Z / 4 \Z$ when $n$ is odd and 
$\Z / 2 \Z \oplus \Z / 2 \Z$ when $n$ is even. In the odd case there is only one 
subgroup of order $2$, but in the even case there are {\em three} of them, two of which 
are indistinguishable up to automorphism of the complete local ring, since their generators 
correspond to conjugate points in the associated Dynkin diagram \cite{J2}. In this case we 
prove the stronger statement that each of the two non-equivalent subgroups arises 
as the class group of a local ring on a surface. 

We also apply our method to general surfaces containing a cone over a set of points. 
For the completed local ring of the vertex of a cone over an elliptic curve, we answer 
Srinivas' question at least modulo a subgroup of co-index 3 ({\em Cf.} Proposition~\ref{ellipticcone} below). The proof uses a combination of the Mather-Yau theorem \cite{MY} from complex analysis and Theorem \ref{gens} with base locus $Y$ consisting of a union of lines which 
naturally generate the subgroup $H$ (or $H'$) along with an embedded point at the vertex. 
We also compute the local Picard groups for a very general surface containing a set of lines 
in general position. 

We work throughout over the field $k = \C$ of complex numbers except as noted. 
While the base locus $Y$ should be projective to apply Theorem \ref{gens}, we often give a local ideal for 
$Y \subset \mathbb A^n$ and apply the theorem to $\overline Y \subset \P^n$. 

\section{Local Picard groups}

In this section we will give some basic properties of the local Picard group and prove 
Theorem \ref{gens} from the introduction. 
Class groups of normal (local) rings have been objects of study in commutative algebra  
for a long time, while in algebraic geometry there has been more interest in the Picard 
group of smooth varieties. In his study of curves on a singular surface $S$ \cite{J1,J2}, 
Jaffe introduced an exact sequence relating the Picard group, class group, and class groups of 
the local rings at the singularities of $S$ (Remark \ref{kadinsky} (a)), 
which he called local Picard groups. We will use the same terminology here.
 

\begin{defn}\label{picloc}{\em Let $p$ be a point on a variety $X$. 
The {\it local Picard group} of $X$ at $p$ is the group $\Picloc p = \APic (\Spec \O_{X,p})$ 
of almost Cartier divisors of the spectrum of the local ring $\O_{X,p}$. 
The {\it formal local Picard group} is $\widehat{\Picloc}\, p = \APic (\Spec \widehat \O_{X,p})$. We will 
write $\Picloc_{X} p$ or $\widehat{\Picloc}_X \, p$ if $X$ is not clear from context.
\em}\end{defn}

\begin{note}{\em 
Jaffe's original definition $\Picloc p = \Pic (\Spec \O_{X,p} \setminus \{p\})$ 
agrees with ours if $X$ is a normal surface and more generally if $X$ is Gorenstein 
in codimension $\leq 1$ and $S_2$ (Serre's condition), when both amount to the divisor 
class group $\Cl (\Spec \O_{X,p}) = \Cl (\O_{X,p})$ \cite[2.15.1]{GD}. 
Therefore when $X$ is normal, we will often write $\Cl (\O_{X,p})$ and $\Cl (\widehat \O_{X,p})$. 
\em}\end{note}

\begin{rmks}\label{kadinsky}{\em We mention some past work involving local Picard groups. 

(a) For an integral surface $S$, Jaffe proved the exact sequence \cite[3.2]{J1}  
\begin{equation}\label{jaffe}
0 \to \Pic S \to \APic S \stackrel{\theta}{\to} \bigoplus_{p \in S} \Picloc p.
\end{equation}
and used it to classify set theoretic complete intersections 
in $\Pthree$ with a cone and another surface. 
Hartshorne extended this sequence to schemes satisfying $G_1$ and $S_2$ \cite[2.15]{GD}.

(b) Lipman computed the formal local Picard groups for pseudo-rational surface singularities 
in terms of the intersection matrix for the exceptional divisors arising from the minimal 
desingularization \cite[$\S 24$]{lipman}.

(c) Parameswaran and Srinivas showed that every completion of a local ring at an isolated local complete intersection 
singularity of a variety over $\C$ is the completion of a unique factorization domain \cite{parasrini}.
 
(d) Hartshorne used a relative version of the local Picard group for degenerating 
families of cubic surfaces in his solution to Zeuthen's problem \cite{Zeuthen}. 
\em}\end{rmks}

On a normal variety the local Picard group can be identified with subgroup 
of the formal Picard group which lift to global Weil divisors: 

\begin{prop}\label{interpret}
Let $p \in X$ be a normal singularity. 
Then $\Cl (\O_{X,p})$ may be identified with the image of the natural 
restriction map $\Cl X \to \Cl \widehat \O_{X,p}$. 
\end{prop}

\begin{proof}
The restriction map factors as $\Cl X \to \Cl \O_{X,p} \to \Cl \widehat \O_{X,p}$ and 
it is well known that the second map is injective 
\footnote{See \cite[Lemma 1.10]{Zeuthen} for a generalization to the non-normal case.}.
Since $X$ is regular in codimension $1$, $\Spec (\O_{X,p} - \{p\})$ is smooth, 
hence any divisor in $\Cl (\O_{X,p}) = \Pic(\Spec \O_{X,p} - \{p\})$ is Cartier 
and the local equation extends to an open affine neighborhood $U$ of $p$, hence to a global 
Weil divisor since the map $\Cl X \to \Cl U$ is surjective. It follows that $\Cl X \to \Cl \O_{X,p}$ 
is surjective and the result follows.
\end{proof}

For very general hypersurfaces $X \subset \P^n_\C$ containing a fixed 
{\it superficial} base locus $Y$ (see the definition following Theorem~\ref{gens} in the introduction), 
our recent work in Noether-Lefschetz theory \cite{BN} combines with Proposition \ref{interpret} 
to describe natural geometric generators for local Picard groups. We restate Theorem \ref{gens} for convenience: 

\begin{thm}\label{method}
Let $Y \subset \P_\C^n$ be superficial and $p \in Y$. 
Then a very general hypersurface $X$ of degree $d \gg 0$ containing $Y$ is normal and 
$\Cl \O_{X,p}$ is generated by supports of the codimension $2$
irreducible components of $Y$. 
In particular, $\dim Y < n-2 \Rightarrow \Cl \O_{X,p} = 0$.
\end{thm} 

\begin{proof}
The class group $\Cl X$ is generated by $\O_X (1)$ and the supports of the codimension 
two components of $Y$ \cite[Theorems 1.1 and 1.7]{BN}, but the images of $\O_X (1)$ and the 
codimension two components of $Z$ which miss $p$ have trivial image in $\Cl \O_{X,p}$. 
\end{proof}

We give an algebraic translation of this result. 
\begin{cor}
Let $I \subset R = \C[x_1,\dots,x_n]$ be an ideal of height $\geq 2$. 
In the primary decomposition $I = \bigcap q_i$ with $q_i$ a $p_i$-primary ideal, 
assume that $q_i \not \subset p_i^2$ for each height two prime $p_i$. 
Then for the very general $f \in I$ and $A = \C[x_1, \dots, x_n]/(f)$, 
the image of $\Cl (A) \hookrightarrow \Cl (\widehat A)$ 
is generated by the height two primes associated to $I$. 
\end{cor}

\section{Power series and completions of unique factorization domains}

In this section we note some useful lemmas about changes of coordinates in power series rings, 
the Mather-Yau theorem, and prove Theorem \ref{ufd} from the introduction. 

\subsection{Coordinate changes}

We will use coordinate changes in $k[[x,y,z]]$ to to simplify equations, enabling us to 
understand the structure of singularities and to calculate local Picard groups.  
For a formal power series ring $R = k[[x_1,\dots,x_n]]$ over a field $k$ with maximal ideal $\fm$, 
a {\it change of coordinates} for $R$ is an assignment $x_i \mapsto x_i^\prime \in \fm$ which induces an 
automorphism of $R$. We include proofs, being unaware of any references.

\begin{rmk}\label{coordchange}{\em An assignment $x_i \mapsto x_i^\prime$ induces an automorphism 
if and only if the induced map $\fm / \fm^2 \to \fm / \fm^2$ is an isomorphism if and only if the matrix 
$A$ of coefficients of linear terms in the $x_i^\prime$ is nonsingular (this is noted by Jaffe \cite[3.2]{J2} when $n=2$). 
\em}\end{rmk}

\paragraph{\em Multiplication of variables by units} The easiest sort of coordinate change, these have the form 
$x_i \mapsto x_i^\prime = u_i x_i$ where $u_i \in R$ is a unit. 

\begin{ex}\label{baby}{\em Consider the equation $f=x^2 + y^2 z - z^{n-1} \in \C [[x,y,z]]$. 
For $\displaystyle a = e^\frac{\pi i}{n-1}$ and $\displaystyle b = e^\frac{-\pi i}{2(n-1)}$, the coordinate 
change $x \mapsto x^\prime = x, z \mapsto z^\prime = az, y \mapsto y^\prime = by$ brings $f$ to the form 
${x^\prime}^2 + {y^\prime}^2 {z^\prime} + {z^\prime}^{n-1}$, the standard form for a $D_n$ singularity.  
\em}\end{ex}

\paragraph{\em Translations} These have the form $x_i \mapsto x_i^\prime = x_i + f_i$ with $f_i \in \fm$. 
In our usage it is often the case that $f_i \in \fm^2$. 

\begin{ex}\label{d3=a3}{\em
Consider the equation $f=x^2 + y^p z + z^2$ in $\C [[x,y,z]]$. Completing the square gives 
$f=x^2 + (z+(1/2)y^p)^2 - (1/4) y^{2p}$ so that after the coordinate change 
$z^\prime = z+(1/2)y^2$ we arrive at $f = x^2 + (z^\prime)^2 - (1/4) y^{2p}$. 
Making the further coordinate change $X=x+iz^\prime, Z=x-iz^\prime$ and $Y=\frac{1}{\sqrt[p]{2}} y$ brings 
$f$ to the form $XY-Y^{2p}$, the standard form of an $A_{2p-1}$ singularity. 
When $p=2$ this shows that a $D_3$ singularity is the same as an $A_3$.
\em}\end{ex}

\paragraph{\em Sequences of coordinate changes} Sometimes it is easier to achieve a coordinate change one step 
at a time, approximating solutions \`a la Hensel. Thus if we have a {\it sequence} of coordinate changes 
$x_i^j \mapsto x_i^{j+1}$ satisfying $x_i^j \equiv x_i^{j+1}$ mod $\fm^{j+1}$, then $x_i^j$ converges to $X_i \in R$ 
and $x_i \mapsto X_i$ is a change of coordinates. The following lemmas illustrate such sequences. 

\begin{lem}\label{comproots} Let $(R,\fm)$ be a complete local domain, and let $n$ be a positive 
integer that is a unit in $R$. If $a_0 \in R$ is a unit and $u \equiv a_0^n$ mod $\fm^k$ 
for some fixed $k > 0$, then there exists $a \in R$ such that $a^n = u$ and $a \equiv a_0$ mod $\fm^k$.
\end{lem}

\begin{proof} 
Using Hensel's method, we construct a sequence $\{a_i\}$ with 
$a_i\equiv a_{i+1} \mod \fm^{(i+1)k}$ and $a_i^n\equiv u \mod \fm^{(i+1)k}$. 
Write $u^{(i)} = u - a_i^n \in \fm^{(i+1)k}$ and let 
$a_{i+1} = a_i + \frac{u^{(i)}}{na_i^{n-1}}$; then $a_{i+1}\equiv a_i\mod \m^{(i+1)k}$ 
and $a_{i+1}^n = a_i^n + u^{(i)} + \tilde{u}, \tilde{u}\in 
\fm^{2(i+1)k} \Rightarrow a_{i+1}^n \equiv u$ mod $\fm^{(i+2)k}$.
\end{proof}

\begin{lem}\label{dnen}
Let $R = k[[y,z]]$ with maximal ideal $\fm \subset R$. 
For integers $a,s,t$ with $s > a > 1, t > a+1$ and $b \in R$ a unit, 
there is a change of coordinates $Y,Z$ such that 
\[
f=y^{a} z + z^{s} - b y^{t} = Y^{a} Z + Z^{s}.
\]
Furthermore $X,Y$ may be chosen so that $y \equiv Y$ mod $\fm^2$ and $z \equiv Z$ mod $\fm^2$.
\end{lem}

\begin{proof}
We produce coordinate changes $y_i, z_i$ such that 
$y_{i+1}\equiv y_i \mod \fm^{i+1}, z_{i+1}\equiv z_i \mod \fm^i$ 
and $f = y_i^a z_{i} + z_i^s -b_i y_{i}^{k_i}$ with 
$k_i\ge i+a+1$ and $b_{i}$ a unit. By hypothesis $y_1=y, z_1=z$ give the base step $i=1$.

For the induction step, let $z_{i+1} = z_i - b_i y^{k_i-a} \equiv z_i\mod \m^{i+1}$ so that
$$f=y_i^a z_{i+1} + (z_{i+1}^s + sb_iz_{i+1}^{s-1}y_i^{k_i-a} + 
\dots + sb_i^{s-1}z_{i+1}y_i^{(s-1)(k_i-a)} + y_i^{s(k_i-a)})$$
$$=y_i^a z_{i+1}\ubm{[1+sb_iz_{i+1}^{s-2}y_i^{k_i-2a} + \dots + 
sb_i^{s-1}y_i^{(s-1)(k_i-2a)}]}{v_{i}} + 
z_{i+1}^s + b_i^s y_i^{s(k_i-a)}$$
where $v_i$ is a unit with lowest-degree term after the leading $1$ is of degree 
$s-2+k_i-2a \ge i$. By Lemma~\ref{comproots}, $v_i$ has an $a^{\rm th}$ root 
$w_i$ that is congruent to $1$ mod $\fm^i$. 
Then $y_{i+1} = w_i y_i \equiv y_i \mod \fm^{i+1}$, so that 
$f = y_{i+1}^a z_{i+1} + z_{i+1}^s - b_{i+1} y_{i+1}^{k_{i+1}}$, 
where $b_{i+1} = -b_i^s w_i^{-s(k_i-a)}$ is a unit and 
$k_{i+1} = s(k_i-a) \ge s(i+1) \ge (a+1)(i+1) \ge a+i+2$, completing the induction.
\end{proof}
 
\paragraph{\em Elementary Transformations} These are a type of translation used by Hartshorne \cite[$\S 4$]{Zeuthen}. 
To simplify an expression $f=xy + g_r z^r + g_{r+1} z^{r+1} + \dots$ with $g_i \in [x,y]$ and $r \geq 1$, 
write $g_r = ax + by$ and let $x^\prime = x+bz^r, y^\prime = y+az^r$. 
Then $x^\prime y^\prime = xy + f_r z^r + ab z^{2r}$ and $f = x^\prime y^\prime + g_{r+1}^\prime z^{r+1} + \dots$ 
for suitable $g_i^\prime$, removing the lowest power of $z$. 
\begin{lem}\label{xyfactor}
Let $R = k[[x,y,z]]$ with maximal ideal $\fm \subset R$. 
For $f \in (x,y)^2 \fm$, there is a change of coordinates $X,Y$ such that 
\[
xy+f = XY.
\]
Furthermore $X,Y$ may be chosen so that $x \equiv X$ mod $\fm^2$ and $y \equiv Y$ mod $\fm^2$. 
\end{lem}

\begin{proof} ({\em Cf.} \cite[I, Ex. 5.6.3]{AG})
Since $f \in \fm^3$, we may write $f = xh_1 + y g_1$ with $h_1, g_1 \in \fm^2$. Now write $x_1 = x+g_1, y_1=y+h_1$ 
so that $xy+f=x_1 y_1 - g_1 h_1$. Since $g_1 h_1 \in \fm^4$, we can continue the process, construting a sequence of 
coordinate changes, which converge is $X,Y$.
\end{proof}

%
%

\begin{lem}\label{Ruiz} Let $\fm \subset \C [[x_1, \dots, x_n]]$ denote the maximal ideal and 
fix $f \in \fm^2$. Then for $g \in \C[[x_1, \dots, x_n]]$, 
$f-g \in \fm \cdot (f_{x_1}, \dots, f_{x_n})^2 \Rightarrow \C[[x_1, \dots x_n]]/(f) \cong \C[[x_1, \dots, x_n]]/(g)$.
\end{lem}

We close with the Mather-Yau theorem from complex analysis \cite{MY}. Recall that for a 
polynomial $f \in \C[x_1, \dots, x_n]$ vanishing at the origin, the {\em moduli algebra} of $f$ 
is 
\[
A(f) = \C\{x_1, \dots, x_n\}/(f,\partial f / \partial x_1, \dots \partial f / \partial x_n)
\]
where $\C\{x_1, \dots, x_n\}$ is the local ring of germs of holomorphic functions at the origin. 

\begin{thm} Isolated singularities at the origin defined by $f, g \in \C[x_1,\dots,x_n]$ 
are complex-analytically isomorphic if and only if $A(f) \cong A(g)$. 
\end{thm}

\begin{rmk}{\em 
Lemma~\ref{Ruiz} is also true with $\C\{x_1, \dots, x_n\}$ in place of $\C[[x_1, \dots, x_n]]$. 
We wish to consider the formal-analytic situation, but we only use the direction of the Mather-Yau 
Theorem that allows us to conclude complex-analytic isomorphism of two rings, which then 
implies formal-analytic isomorphism.} 
\end{rmk}

\subsection{Applications to local Picard groups} 
Given a complete local ring $A$, what is the nicest ring $R$ having completion $A$? 
Heitmann shows that $A$ is the completion of a UFD if and only if $A$ is a field, $A$ is a DVR, or $A$ 
has depth $\geq 2$ and no integer is a zero-divisor of $A$ \cite{heit1}, however his constructions 
need not lead to excellent rings. Loepp shows with minimal hypothesis that $A$ is the 
completion of an {\em excellent} local ring \cite{loepp}, though his construction need not produce a UFD. 
Using geometric methods, Parameswaran and Srinivas show that the completion of the local ring at an isolated 
local complete intersection singularity is the completion of a UFD which is the local ring for a variety \cite{parasrini}. 
We prove the same for normal hypersurface singularities over $\C$ which need not be isolated.  

\begin{thm} Let $A=\C[[x_1, \dots x_n]]/f$, where $f$ is a polynomial defining a 
variety $V$ which is normal at the origin. 
Then there exists an algebraic hypersurface $X \subset \P^n_\C$ 
and a point $p \in X$ such that $R=\O_{X,p}$ is a UFD and $\widehat R \cong A$. 
\end{thm}

\begin{proof} 
The hypersurface $F: \{f=0\}$ has singular locus $D$ given by the ideal $(f) + J_f$, 
where $J_f = (f_{x_1}, \dots, f_{x_n})$ is the ideal generated by the partial derivatives of $f$.
The closed subscheme $Y$ defined by the ideal $I_Y=(f, f_{x_1}^3, \dots, f_{x_n}^3)$ is supported on $D$ 
because $I_Y \subseteq I+J_f \subseteq \sqrt{I_Y}$, hence $\codim Y = \codim D \geq 3$ by normality of $F$. 
The very general hypersurface $X$ containing $Y$ satisfies $\Cl \O_{X,p} = 0$ by Theorem \ref{gens}, 
so $\O_{X,p}$ is a UFD. The local equation of $X$ is $g = f + a_1f_{x_1}^3 + \dots + a_n f_{x_n}^3$ for units $a_i$ and $f-g\in J_f^3\subseteq \fm J_f^2$, so 
$\widehat \O_{X,p} = \C[[x_1, \dots, x_n]]/(g) \cong \C[[x_1, \dots, x_n]]/(f) = A$ by Lemma \ref{Ruiz}.
\end{proof}

\begin{ex}{\em For an {\em isolated} singularity, one can give the following proof based on the Mather-Yau theorem. 
The ideal $I=(f, f_{x_1}, \dots, f_{x_n})$ generated by $f$ and its partial derivatives defines a $0$-dimensional 
scheme $Y$ supported at the origin, hence $(x_1, \dots, x_n)^N \subset I$ for some $N > 0$. The scheme $Z$ defined by 
$(f, x_1^{N+2}, \dots, x_n^{N+2})$ is also supported at the origin $p$, so by Theorem \ref{gens} the very general 
surface $S$ containing $Z$ satisfies $\Cl \O_{S,p} = 0$. The local equation of $S$ has the form 
$g = f+\sum a_i x_i^{N+2}$ for units $a_i$ in the local ring $\O_{\P^n,p} \cong \C[x_1, \dots, x_n]_(x_1, \dots, x_n)$. 
Observe that $J = (g,g_{x_1}, \dots, g_{x_n}) \subset I$ because $g-f$ and its partials lie in 
$(x_1, \dots x_n)^{N+1} \subset I$. These ideals are equal because the induced map $J \to I/(x_1, \dots x_n)I$ 
is obviously surjective, therefore so is the map $J \to I$ by Nakayama's lemma. It follows that $I = J$ in the 
ring $\C{x_1, \dots, x_n}$ of germs of holomorphic functions as well, so $f$ and $g$ define 
(complex)-analytically isomorphic singularities; this isomorphism lifts to a formal-analytic isomorphism.
We have thus produced a UFD, namely $\O_{S,p}$, whose completion is isomorphic to $A$.
\em}\end{ex}

\section{Local Picard groups of Rational Double Points}

As noted in the introduction, Mohan Kumar showed that for $\rdA_n$ and $\rdE_n$ double points 
on a rational surface, the analytic isomorphism class determines the algebraic isomorphism 
class of the local ring with the three exceptions $\rdA_7, \rdA_8, \rdE_8$, for which there are 
two possibilities each. Regarding the question of Srinivas mentioned in the introduction, this means that there 
is one (or two) possibilities for the inclusion $\Cl A \hookrightarrow \Cl \widehat A$ for the 
corresponding local rings. Our main goal in this section is to prove Theorem \ref{applic2} 
from the introduction, which says that without the rationality hypothesis, {\em all} subgroups are possible. 
For these singularities the group $\Cl \widehat A$ has been known for a long time \cite{artin,lipman}, 
they are as follows: $\Z/(n+1)\Z$ for an $\rdA_n$; $\Z/4\Z$ for a $\rdD_n$ with $n$ even; $\Z/2\Z \oplus \Z/2\Z$ 
for a $\rdD_n$ with $n$ odd; $\Z/2\Z$ for an $\rdE_6$; $\Z/3\Z$ for an $\rdE_7$; and $0$ for an $\rdE_8$. 
For each subgroup of these we will construct a surface $S \subset \Pthree$ and a point $p \in S$ for 
which $\Cl \O_{S,p}$ realizes the subgroup given. We begin with some generalities on rational double points.

\subsection{Rational double points} In this section we review rational double points, including their 
minimal desingularizations, the class groups of their completed local rings and how to compute curve classes. 

\subsubsection{$\rdA_n$ Singularities} These have the form $\Spec(R)$ with $R = k [[x,y,z]]/(xy-z^{n+1})$ and 
have a well-known minimal resolution \cite[5.2]{Zeuthen}: The $\rdA_1$ resolves in a single blow-up with one 
rational exceptional curve having self-intersection $-2$; the $\rdA_2$ resolves in one blow-up but with two 
$(-2)$-curves meeting at a point; For $n \ge 3$, blowing up with new variables $x_1 = x/z, y_1 = y/z$ gives 
two exceptional curves $E_{1}$ defined by $(x_1,z)$ and $E_{n}$ defined by $(y_1,z)$, meeting transversely 
at an $\rdA_{n-2}$ at the origin. Blowing up and continuing inductively, the singularity unfolds with 
exceptional divisor a chain of rational $(-2)$-curves $E_i$ with $E_i$ meeting $E_{i+1}$ transversely 
for $1 \leq i < n$, giving the corresponding Dynkin diagram
$$\entrymodifiers={!! <0pt, .6ex>+}\xymatrix{
\underset{E_1}\circ\ar@{-}[r] & \underset{E_2}\circ\ar@{-}[r] & \dots & \underset{E_n}{\circ}\ar@{-}[l] }
$$ 
To compute the class of a curve $C$ in the formal Picard group, form the strict transform 
$\widetilde C$ in the resolution described above and set 
$(a_1, \dots a_n) = (\widetilde C . E_1, \dots \widetilde C . E_n)$; then the class of $C$ is 
$\sum i a_i \in \Z / (n+1) \Z \cong \APic (\Spec R)$. 

\begin{ex}\label{image}{\em We compute the classes of some curves in $\APic (\Spec R) \cong \Z / (n+1) \Z$. 

(a) The curve $C$ with ideal $(x,z)$ has class $1$ because $\widetilde C$ meets only $E_1$ transversely. 

(b) The class of $D$ given by ideal $(y,z)$ has class $-1$ because it meets only $E_n$.

(c) For $1 \leq k \leq n$ and unit $a$, the curve with ideal $(x-a z^{n-k+1}, y-a^{-1} z^k)$ has 
class $k$ \cite[Remark 5.2.1]{Zeuthen}. Adjusting the variables by units, we will often apply this 
to the ideal $(x-uz^{n-k+1},y-vz^k)$ in the ring $k [[x,y,z]]/(xy-uvz^{n+1})$ with $u,v$ units. 
\em}\end{ex}

\subsubsection{$\rdD_n$ Singularities} For $n \geq 4$, these singularities are given by the spectrum 
of the ring $R = k [[x,y,z]]/(x^2+y^2 z + z^{n-1})$. Again the minimal desingularization and class 
group are well-known \cite[$\S 14, \S 17$]{lipman}. For $n \geq 6$, blowing up the singularity gives 
a single rational $(-2)$-exceptional curve $E_2$ with an $\rdA_1$ and a $\rdD_{n-2}$ singularity on it. 
Blowing up the $\rdA_1$ gives the exceptional divisor $E_1$ meeting $E_2$ transversely and blowing up 
the $\rdD_{n-2}$ produces another rational $(-2)$ exceptional curve $E_4$ which meets $E_2$ at the 
$A_1$ singularity contained in $E_4$: blowing up the $A_1$ produces $E_3$ which connects $E_2$ and $E_4$. 
Continuing in this fashion we obtain a chain of rational $(-2)$-curves until we get to the $\rdD_5$ or $\rdD_4$. 
The $\rdD_4$ resolves by a sequence of two blow-ups, the first resulting in a single exceptional curve 
$E_{n-2}$ with three type-$\rdA_1$ singularities along it, one lying on $E_{n-4}$ - blowing up this 
$\rdA_1$ gives $E_{n-3}$ connecting $E_{n-4}$ to $E_{n-2}$ and blowing up the other two gives 
$E_{n-1}$ and $E_{n}$. The resolution for the $\rdD_5$ first gives a single exceptional curve with an 
$\rdA_1$ and an $\rdD_3$ along it, but $\rdD_3 \cong \rdA_3$ by Example \ref{d3=a3} and is resolved as above. 
Thus the exceptional divisor consists of $n$ rational $(-2)$-curves, the first $n-2$ in order forming a 
chain and the last two meeting $E_{n-2}$ transversely forming a split tail. The corresponding Dynkin diagram is
$$\entrymodifiers={!! <0pt, .6ex>+}\xymatrix{ & & &  \overset{E_{n-1}}{\circ}\ar@{-}[d] & \\
\underset{E_1}{\circ} \ar@{-}[r] & \underset{E_2}\circ\ar@{-}[r] & \dots & \underset{E_{n-2}}{\circ}\ar@{-}[l] \ar@{-}[r]  & \underset{E_n}{\circ}  
}
$$
To compute the class of a curve $C$ in the formal Picard group, form the strict transform $\widetilde C$ in 
the resolution described above and set $(a_1, \dots a_n) = (\widetilde C . E_1, \dots \widetilde C . E_n) \in \Z^n$ 
as before. We can again calculate $\APic \Spec R$ as the quotient of the free group $\Z^n$ on the generators 
with relations given by the intersection numbers ~\cite[$\S 14,17$]{lipman}: 
\[ 
-2u_1+u_2, u_1-2u_2+u_3, \dots, u_{n-3}-2u_{n-2} + u_{n-1} + u_n, u_{n-2} - 2u_{n-1}, u_{n-2} - 2u_n  
\] 
so that $u_2=2u_1, u_3= 3u_1, \dots, u_{n-2} = (n-2) u_1$; from the last three relations deduce 
that $(2n-2) u_1 = 2(u_{n-1}+u_{n-2}) = (2n-4) u_1$, so that $2u_1=0$. Thus if $n$ is odd, 
$u_{n-1} + u_n = 0$ and $2u_{n-1}=2u_n = u_1$, giving $\APic (\Spec R) \cong \Z/4\Z$; if $n$ is even, 
$u_{n-1}+u_{n-2} = u_1$ and $2u_{n-1}=2u_n = 0$, so that $\APic (\Spec R) \cong \Z/2\Z \oplus \Z/2\Z$. 

\subsubsection{$\rdE_n$ Singularities} The local equation for the $\rdE_6$ (resp. $\rdE_7, \rdE_8$) type 
singularity is $x^2+y^3+z^4$ (resp. $x^2+y^3+yz^3, x^2+y^3+z^5$) with class group $\Z / 3 \Z$ (resp. $\Z / 2 \Z, 0$). 
The exceptional curves for the minimal desingularization of an $\rdE_n, n=6,7,8$ singularity correspond to the 
Dynkin diagram 
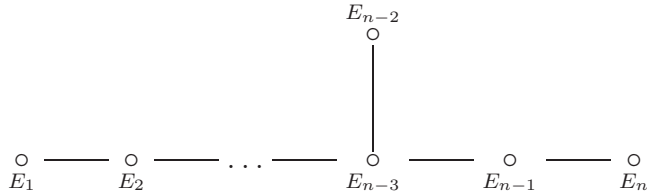
\[ \entrymodifiers={!! <0pt, .6ex>+}
\xymatrix{ & & &  \overset{E_{n-2}}{\circ}\ar@{-}[d] & \\
\underset{E_1}{\circ} \ar@{-}[r] & \underset{E_2}\circ\ar@{-}[r] &
\dots & \underset{E_{n-3}}{\circ}\ar@{-}[l] \ar@{-}[r]  & 
\underset{E_{n-1}}{\circ} \ar@{-}[r]  & \underset{E_n}{\circ}} 
\]
and the resolutions are straightforward to calculate.

\subsubsection{The fundamental cycle and smooth representatives for local Picard groups}
In this section we use the fundamental cycle to prove that at any 
rational double point, each class in the local Picard group can be 
lifted to a locally smooth curve.

If $p \in S$ is a rational double point with minimal desingularization $X \to S$, the {\em fundamental cycle} 
$\xi_0$ is the unique minimal effective nonzero exceptionally supported divisor on $X$ having nonpositive 
intersection with each exceptional curve \cite[pp. 131-2]{artin}. With the conventions used in describing 
the minimal desingularizations above, the fundamental cycles for each singularity type are given in Table \ref{funcycle}.

\begin{table}[ht]
    \caption{Fundamental cycle for rational double points.}
    \label{funcycle}
    \begin{tabular}{l|lll}
	Type   & $\xi_0$  \\
	\hline 
	$\rdA_n$ & $E_1+E_2+\dots+E_n$   \\
	\hline 
	$\rdD_n$ & $E_1+2E_2+2E_3+\dots+2E_{n-2}+E_{n-1}+E_n$   \\
	\hline 
	$\rdE_6$ & $E_1+2E_2+3E_3+2E_4+2E_5+E_6$   \\
	\hline
	$\rdE_7$ & $E_1+2E_2+3E_3+4E_4+2E_5+3E_6+2E_7$  \\
	\hline
	$\rdE_8$ & $2E_1 + 3E_2 + 4E_3 + 5E_4 + 6E_5 + 3E_6 + 4E_7 + 2E_8$ \\
    \end{tabular}
\end{table}

\begin{rmk}\label{admiss}{\em 
The fundamental cycle has the property that for a non-exceptional smooth curve $D$ on $X$, $D.\xi_0$ is equal to the multiplicity of $\pi(D)$ on $S$~\cite[Prop. 5.3]{B}.  For a given rational singularity, the exceptional divisors $E_i$ with coefficient $1$ 
in the fundamental cycle $\xi_0$ are said to be {\em admissible} \cite{J2}. 

We say that a divisor on the minimal desingularization of a rational double point has {\em small intersection with the fundamental cycle} if its $n$-tuple of intersection numbers consists of all $0$s except possibly a $1$ at a single admissible curve. Using the generators and relations for the intersection numbers of the exceptional curves, it is easy enough to verify that the admissible curves represent each nonzero element of the class group for each singularity. Thus, since the relations are generated by the intersection multiplicities of the exceptional curves themselves, it is clear that on the minimal desingularization, it is immediate that any divisor on the desingularization differs from one having small intersection with the fundamental cycle by some exceptionally supported divisor.
}\end{rmk} 

We begin with a combinatorial lemma that shows that in Remark~\ref{admiss}, if all of the intersection numbers are positive to begin with, the necessary adjustment is by an {\em effective} divisor. This lemma generalizes and makes a small correction to~\cite[Prop. 5.5]{B}.

\begin{lem}\label{grunge}
Let $X \to S$ be a minimal desingularization of a rational double point $p \in S$ 
with exceptional divisors $E_1, \dots E_n$ and fix a divisor $F$ on $X$ such that 
$F.E_i \geq 0$ for all $i$. Then there exists an effectively supported divisor $G$ 
on $X$ such that $(F+G).E_i = 0$ for all but at most one $i$; if such $i$ exists, 
then $(F+G).E_i=1$ and $E_i$ appears with coefficient $1$ in the fundamental cycle $\xi_0$.
\end{lem}

\begin{proof} Letting $s_i = F.E_i$, we are done if $\sum s_i \leq 1$. 
Assuming $\sum s_i > 1$, we show how to add sums of the $E_i$ to achieve the statement in each case.

For an $\rdA_n$-singularity let $j$ (resp. $k$) be the least (resp. greatest) index $i$ with $s_i > 0$. 
If $j < k$, then adding $E_j+E_{j+1} + \dots +E_k$ to $F$ decreases $s_j, s_k$ by $1$, increases 
$s_{j-1}, s_{k+1}$ by $1$ (if these exist) and has no effect on the remaining $s_i$. 
Repeat until $j=1$ or $k=n$ at which point $\sum s_i$ decreases. 
If $1 < j=k < n$ and $s_j \geq 2$, then adding $E_j$ reduces to the previous case. 
Eventually $\sum s_i = 1$ and we are done and the final statement is clear because 
the coefficient of each $E_i$ in $\xi_0$ is $1$.

For a $\rdD_n$ singularity, note that adding $E_{n-1}$ (resp. $E_n$) decreases $s_{n-1}$ (resp. $s_n$) by $2$ 
and increases $E_{n-2}$ by $1$, so we may assume $0 \leq s_{n-1}, s_n \leq 1$. We may assume $s_i = 0$ for $1 < i < n-1$. 
If this is not so, let $k$ be the largest $i$ in this range with $s_k > 0$. 
Adding $s_k (E_k + 2E_{k+1} + \dots + 2E_{n-2}+E_{n-1}+E_n)$ increases $s_{k-1}$ by $s_k$, decreases $s_k$ to $0$ 
and leaves the remaining $s_i$ fixed, thus $k$ decreases and we may continue until $k=1$. 
Adding $2 E_1 + \dots + 2 E_{n-2} + E_{n-1} + E_n$ decreases $s_1$ by $2$ and fixes the remaining $s_i$, 
so we may assume $0 \leq s_1, s_{n-1}, s_n \leq 1$ (the rest are zero). 
We are done if at most one of these is nonzero, otherwise adding 
$s_1 E_1 + s_2 E_2 + E_3 + \dots + E_{n-2} + s_{n-1} E_{n-1} + s_n E_n$ 
switches $s_1, s_{n-1}, s_n$ from 0 to 1 or 1 to 0 and fixes the rest, finishing the proof.
 
For an $\rdE_6$, note that adding $\xi_0$ reduces $s_4$ while fixing the rest, so we may assume 
$s_4=0$ as needed. Now applying the $\rdA_5$ strategy to remaining chain reduces 
to the case that at most one of $\{s_i\}_{i \neq 4}$ is $1$ while the rest of these are zero; 
if $s_1 = 1$ or $s_6=1$, we are done; if $s_2=1$, adding $E_1+2E_2+2E_3+E_4+E_5$ sets $s_2=0$ and 
$s_6=1$ (the case $s_5=1$ is similar); if $s_3=1$, adding $E_1+2E_2+3E_3+E_4+2E_5+E_6$ decreases 
$s_3$ by $1$ and increases $s_4$ by $1$. 

For an $\rdE_7$ singularity, applying the $\rdD_6$ strategy to $E_1, \dots, E_6$ reduces to the case 
$s_2=s_3=s_4=0$ and at most one of $s_1, s_5, s_6$ is $1$; $s_7$ will likely be increased in the process, 
but adding $\xi_0$ decreases $s_7$ by one while fixing the rest, so we may assume $s_7=0$. 
If $s_1=1$, we are done; if $s_5=1$, adding $E_2+2E_3+3E_4+2E_5+2E_6+E_7$ sets $s_5$ to $0$ while increasing 
$s_1$ to $1$; if $s_6=1$, adding $2E_1+4E_2+6E_3+8E_4+4E_5+6E_3+3E_7$ sets $s_6=0$ while fixing the rest.

Finally, for an $\rdE_8$, the exceptional curves $E_2, \dots, E_8$ form an $\rdE_7$ singularity; as above 
we can add exceptional curves to reach a point where $s_2 = 0$ or $1$, $s_3 = s_4 = \dots = s_8 = 0$, 
and $s_1$ is still positive. Adding $\xi_0$ has the effect of decreasing $s_1$ by $1$ and fixing the other $s_j$, 
so repeated addition of $\xi_0$ brings us to the situation where all the $s_j$ are $0$, in which case we 
are finished, or $s_2=1$ and all other $s_j$ are $0$. In this latter case add 
$3E_1+6E_2+8E_3+10E_4 + 12E_5 + 6E_6 +  8E_7 + 4E_8$, which has the effect of decreasing $s_1$ to $0$ 
and leaving all the other $s_j$ constant. 
\end{proof}

\begin{prop} Let $p$ be a rational double point on a projective surface 
$S$ and let $x \in \Cl (\O_{S,p})$. Then there is an effective divisor 
$D \in \Cl S$ smooth at $p$ restricting to $x$. 
\end{prop}

\begin{proof} 
The class $x$ lifts to $C \in \Cl S$ by Proposition \ref{interpret}. We may assume that $C$ is effective 
after adding a high multiple of $\O_S (1)$, since $\O_S (1)$ has trivial restriction to $\Cl (\O_{S,p})$.
Let $\pi: X \ra S$ be the minimal resolution of singularities with irreducible exceptional curves $E_i$ 
and let $\widetilde{C}$ be the strict transform of $C$ on $X$. By Lemma~\ref{grunge} there exists 
an effective exceptionally supported divisor $G$ on $X$ such that $(\widetilde{C} + G).E_i=0$ for 
all but at most one of the $E_i$, and if such an $E_i$ exists, then $(\widetilde{C}+G).E_i=1$ 
and $E_i$ is admissible.

Setting $A = \widetilde{C}+G$ we claim that no exceptional curve $E_i$ is a fixed component of the 
linear system $|H^0(\O_X(A))|$. Indeed, the composite map $\O_X \to \O_X (A)|_{E_i}$ in the upper 
right corner of the diagram 
\[
\begin{array}{ccccccccc}
0 & \ra & \O_X(-E_i) & \ra & \O_X & \ra & \O_{E_i} & \ra & 0 \\ 
& & \downarrow & & \downarrow & & \downarrow & & \\ 
0 & \ra & \O_X(A-E_i) & \ra & \O_X(A) & \ra & \O_X(A)|_{E_i} & \ra & 0 
\end{array} 
\] 
is nonzero on global sections, the first map being surjective and the second injective. 
Therefore the map $H^0 (\O_X(A)) \ra H^0(\O_X(A)|_{E_i})$ is also nonzero, which implies that 
$H^0(\O_X(A-E_i)) \ra H^0(\O_X (A))$ is not surjective, verifying the claim.

Now consider the general member $\widetilde D \in |H^0 (\O_X (A))|$. 
If $A.E_i=0$, then $\widetilde D$ misses $E_i$ because $E_i$ is 
not a fixed component. Therefore if $A.E_i=0$ for all $i$, then 
$D = \pi (\widetilde{D})$ misses $p$ and we are done. Otherwise there is at 
most one $E_i$ for which $A.E_i=1$ and $|H^0 (\O_X(A))|$ may have {\em one} fixed 
point $P_0 \in E_i$. By Bertini's theorem the only singular point for the general 
member of the linear system on the exceptional locus can be $P_0$, but even in 
this case $\widetilde{D}$ is smooth at $P_0$ because $\widetilde{D}.E_i=1$. 
Therefore $\widetilde{D}$ is smooth along $E = \bigcup E_i$ and meets $\xi_0$ with 
multiplicity one, so $\mult_p (\pi(\widetilde{D})) = \widetilde{D}.\xi_0=1$ \cite[Prop. 5.4]{B} 
and $D = \pi(\widetilde{D})$ is smooth at $p$. 
Moreover $D$ has the same class $x \in \Cl (\O_{S,p})$ as $C$ because 
$\widetilde{D}$ and $\widetilde{C}$ differ by a sum of exceptional divisors. 
\end{proof}


\subsection{Characterization of Class Groups} We prove Theorem \ref{applic2}. 

Theorem~\ref{ufd} dispenses with the UFD case $H=0$, so assume $H \neq 0$; further, 
we may assume $T \neq \rdE_8$ because there $\widehat{\Picloc}\, = 0$. If $T = \rdE_6$, 
we take the affine quartic surface $S$ with equation $x^2+y^3+z^4=0$, for which the $\rdE_6$ 
singularity at the origin $p$ is apparent. The smooth curve $C$ with ideal $(x-i z^2, y)$ lies 
on $S$ and passes through $p$. If $C$ restricts to $0$ in $\Picloc p$, then $C$ is Cartier on 
$S$ at $p$, impossible because $C$ is smooth at $p$ while $S$ is not: thus $C$ defines a non-zero 
element in $\Picloc p$ which generates all of $\widehat{\Picloc}\, p \cong \Z / 3 \Z$, the only 
non-trivial subgroup. Similarly for $T = \rdE_7$ we take the affine quartic $S$ defined by the 
equation $x^2+y^3+yz^3$ which contains the smooth $z$-axis $C$ passing through the origin $p$. 

The non-trivial subgroups $H$ of the formal Picard groups at $\rdA_n$ and $\rdD_n$ singularities 
are more difficult, requiring the construction of a $1$-dimensional base locus inducing in its 
general surface the appropriate singularity and subgroup. Starting with the $\rdA_n$ case, 
let $H = \cyc{k} \subset \Z/(n+1)\Z$ be a non-trivial subgroup for $k|(n+1)$. 
If $k=1$, the surface with equation $xy-z^{n+1}$ has an $\rdA_n$-singularity at the origin 
and the line $L$ with ideal $(x,z)$ has class $1$ by Example \ref{image} (a), generating 
the full group $Z/(n+1)\Z$. Taking $q=k$ and $m-1 = (n+1)/k$ in Proposotion \ref{step1} below gives the 
nonzero proper subgroups.

\begin{prop}\label{step1} 
Let $Z \subset \Pthree_\C$ be the subscheme with ideal 
$\I_{Z} = (x^2,xy,x z^q - y^{m-1}, y^m)$ for $m \geq 3$. 
Then the very general surface $S$ containing $Z$ has an $\rdA_{(m-1)q-1}$ singularity at 
$p=(0,0,0,1)$ and $\Picloc p \cong \Z / (m-1) \Z$ is generated by $C$.    
\end{prop}

\begin{proof}
For appropriate units $a,b,c \in \O_{\Pthree,p}$ the general surface $S$ containing $Z$ has equation 
\begin{equation} 
xy + ax^2 + b y^{m-1} + c y^m -b x z^q = x \ubm{(y+ax)}{y_1} + \ubm{(b+cy)}{u} y^{m-1} - b x z^q.
\end{equation}
With $y_1 = y+ax$ and unit $u=b+cy$ as above, we may write 
\begin{equation}\label{zwei}
x {y_1} + u ({y_1}-ax)^{m-1} - b x z^q = {x_1}{y_1} + c({y_1}+{x_1})^{m-1}-d {x_1} z^q
\end{equation}
for new units $c,d$ after setting $x_1 = -ax$ and multiplying by $-1/a$. 
If $m=3$, the first two terms are a homogeneous quadratic form: for general $a,b,c$, this 
factors into two linear terms and making the corresponding change of variable brings the equation to 
the form $XY + (AX+BY) z^q$ for units $A,B$. 

For $m \geq 4$, expand the $(m-1)$st power in equation (\ref{zwei}) as 
\[
c({y_1}+{x_1})^{m-1} = {y_1} \ubm{[c {y_1}^{m-2}]}{g_1} + 
{x_1} \ubm{[c ((m-1) {y_1}^{m-2} + \dots + {x_1}^{m-2})]}{h_1}
\]
and set $x_2 = {x_1} + g_1, y_2 = {y_1} + h_1$, when equation (\ref{zwei}) becomes 
\begin{equation}
x_2 y_2 - g_1 h_1 - d x z^q
\end{equation} 
with $g_1 h_1 \in (x,y)^{2m-4} \subset (x,y)^{m}$. Applying Lemma \ref{xyfactor}, 
we make another change of variables from $x_2,y_2$ to $X,Y$ for which $x_2 y_2 - g_1 h_1 = XY$. 
Looking at the last term, notice that $d x=d (x_2 - c y^{m-2}) = d (x_2 - c (y_2 - h_1)^{m-2})$ 
with $h_1 \in (x,y)^{m-2}$; extracting the multiples of $x_2$ in the resulting power series this may 
be written $A_1 x_2 + B_1 y_2^{m-2}$ with $A_1, B_1$ units. Switching to the variables $X,Y$ and noting 
that $x_2 \equiv X$ mod $(x,y)^{m-1}$ and $y_2 \equiv Y$ mod $(x,y)^{m-1}$ by the proof of Lemma \ref{xyfactor} 
($g_1 h_1 \in (x,y)^{m}$), this may be written $AX + BY^{m-2}$ with $A,B$ units. Thus $S$ has local equation
\[
XY - (AX + BY^{m-2}) z^q
\]
with units $A,B$. Setting $Y_2 = Y-Az^q$, we obtain $X Y_2 - B(Y_2+Az^q)^{m-2} z^q$.  
Multiplying out the $(m-2)$nd power, the second term may be written $Y_2 L + B A^{m-2} z^{(m-1)q}$ with 
\[
L = B (Y_2^{m-3} + A(m-2) Y_2^{m-4} z^q + \dots + A^{m-3}(m-2)z^{(m-3)q}) z^q.
\]
Setting $X_2 = X - L$ gives the form 
\[
X_2 Y_2 - A^{m-2} B z^{(m-1)q},
\] 
showing that $S$ has an $\rdA_{(m-1)q-1}$ singularity.

We now follow the ideal $(x,y)$ through its coordinate changes: 
\[
(x,y) = (x_1, y_1) = (x_2,y_2) = (X, Y)=(X, Y_2 + A z^q) = (X_2 + L, Y_2 - A z^q).
\] 
Working modulo $Y_2 - A z^q$, replacing $Y_2$ with $A z^q$ reduces $L$ to 
\[
B A^{m-3} [1 + (m-2) + \binom{m-2}{2} \dots + (m-2)] z^{(m-2)q} = (2^{m-2} - 1) B A^{m-3} z^{(m-2)q},
\]
so the final form for our ideal is 
$(X_2 + (2^{m-2} - 1) BA^{m-3}z^{(m-2)q}, Y_2-Az^q)$, which has class $q$ in 
$\Cl \widehat \O_{S,p} \cong \Z/(m-1) \Z$ by Example \ref{image} (c).
\end{proof}

\begin{rmk}{\em The scheme $Z$ above is the general form for a locally Cohen-Macaulay 
$m$-structure on a smooth curve for which the $m-1$-substructure is contained in a 
smooth surface. At least if $k < (n+1)/2$, one can take as base locus $Z$ the smooth curve 
with an embedded point given by the ideal 
\[
I_Z = (xy-z^{n+1}, x^3-z^{3(n-k+1)}, y^r-z^{n+1}, xz^{n+1}-z^{2n-k+2}, yz^{n+1}-z^{n+k+1})
\]
with $r = (n+1)/k$, when one can show that the origin is an $\rdA_n$ singularity and the 
supporting smooth curve with ideal $(x,y-z^k)$ has class $k$. Unfortunately we know of no 
locally Cohen-Macaulay curves $Z$ which give rise to $\rdD_n$ singularities in this fashion, 
so $Z$ will consist of a smooth curve and an embedded point in the remainder of the proof.  
}\end{rmk}

The $\rdD_n$ singularities have formal Picard group $\widehat{\Picloc}\, p \cong \Z / 4 \Z$ 
or $\Z / 2 \Z \oplus \Z / 2 \Z$ and we first produce the subgroups of order two as local Picard groups. 
The scheme $Z$ defined by $I_Z = (x^2,y^2z, z^{n-1}, xy^n)$ consists of the line $L: x=z=0$ 
and an embedded point at the origin $p$, hence the very general surface $S$ containing $Z$ 
has local Picard group $\Picloc p$ generated by $L$ by Theorem \ref{gens} 
and $L \neq 0$ in $\Picloc p$ because $L$ is not Cartier at $p$, for $L$ is smooth at $p$ while $S$ is not. 
The local equation of $S$ has the form $ax^2 + by^2z + z^{n-1} + cxy^n$ with units 
$a, b, c \in \O_{\Pthree,p}$. Taking square roots of $a,b$ in $\widehat \O_{\Pthree,p}$ 
the equation becomes 
\[ 
x^2 + y^2z + z^{n-1} + cxy^n=(\ubm{x+\frac{c}{2}y^n}{x_1})^2 + y^2z + z^{n-1} - \frac{c^2}{4} y^{2n} 
\] 
and applying the coordinate change of Lemma \ref{dnen} exhibits the $\rdD_n$ singularity.

To determine the class of $L$, we blow up $S$ at $p$. The local equation of $\widetilde S$ on the 
patch $Y=1$ is $aX^2 + byZ + y^{n-3}Z^{n-1} + cXy^{n-1}$ and the strict transform $\widetilde L$ 
has ideal $(X,Z)$, hence $\widetilde L$ meets the exceptional curve at the new origin of this patch, 
which is the $\rdA_1$ singularity whose blow-up will produce the exceptional divisor $E_1$; 
Resolving this $\rdA_1$ shows that $\widetilde L$ meets $E_1$ but not $E_2$, so $L$ gives the class 
$u_1$ defined above. In particular, $2L = 0$ and $\Picloc p \cong \Z / 2 \Z$. 

When $n$ is odd, $\cyc{u_1}$ is the only subgroup of order $2$ in $\widehat{\Picloc}\, p$ and we are 
finished; when $n$ is even -- at least when $n \geq 6$ -- there are three such subgroups 
$\cyc{u_1}, \cyc{u_{n-1}}$ and $\cyc{u_n}$, the last two being distinguishable from the first, 
but not from one another, as they correspond to the two exceptional curves in the final blow-up. 
We therefore construct a $\rdD_n$ singularity, $n$ even, such that its local Picard group is 
generated by $u_{n-1}$ or $u_n$. 

Write $n=2r$ and define $Z$ by the ideal $I_Z = (x^2, y^2z-z^{2r-1}, y^5-z^{5r-5})$. 
The last two generators show that $y \neq 0 \iff z \neq 0$ along $Z$, when $I_Z$ is 
locally equal to $(x^2, y - z^{r-1})$, thus $Z$ consists of a double structure on the 
smooth curve $C$ with ideal $(x,y-z^{r-1})$ and an embedded point at the origin $p$. 
Therefore if $S$ is a very general surface containing $Z$, then as before $C$ generates 
$\Picloc p$ and is nonzero. The local equation for $S$ has the form 
$ax^2 + y^2z-z^{2r-1} + by^5-bz^{5r-5}$ for units $a,b \in \O_{\Pthree,p}$. 
Passing to the completion and adjusting the variables by appropriate roots of units, the equation 
of $S$ becomes $x^2+y^2z+z^{2r-1}+y^5$. Changing variables via Lemma \ref{dnen}, the equation becomes 
$x^2 + Y^2 Z + Z^{2r-1}$ and we see the $\rdD_{2r}$-singularity. 

To determine the class of $C$ in $\widehat{\Picloc}\, p$ we return to the original form of the 
equation for $S$ and blow up $p$. Following our usual conventions, look on the patch $Z=1$, 
where $S$ has equation $X^2 + aY^2z - az^{2r-3} + bY^5z^3 - bZ^{5r-7}$. As long as $2r-3 \ge 3$, 
that is, $r\ge 3$, we can take advantage of our knowledge of the original singularity, together 
with the fact that this expression is a square mod the third power of the maximal ideal at the origin, 
to conclude that it is this new surface has a $\rdD_{2r-2}$ singularity there; the strict transform 
$\tilde{C}$ of $C$ has ideal $(X,Y-z^{r-3})$ and so passes through the origin transversely to the 
exceptional curve $E_1$ given by $(X,z)$. Now, on the full resolution of singularities $\tilde{C}$ 
maps to a smooth curve and therefore meets exactly one admissible $E_i$, that is, either $E_1$, $E_{n-1}$, 
or $E_n$. By the transversality noted above, however, it does not meet $E_1$ in the next blowup. 
We conclude that it meets one of the other curves and so we have produced the desired subgroup.

Finally, we need an example for which $\Picloc p = \widehat{\Picloc}\, p$. To this end consider 
the surface $S$ defined by $x^2 + y^2 z - z^{n-1}$, which has a $\rdD_n$ singularity at the origin 
by Example \ref{baby}. As above the curve with ideal $(x,z)$ corresponds to $u_1$, so it suffices 
in all cases to find a curve on this surface that gives the element $u_n$ (or $u_{n-1}$). 

For even values of $n$, the curve $(x,y-z^\frac{n-2}{2})$ gives the other generator by an argument 
analogous to, but much easier than, the one for $n$ even in the previous case considered.

For odd values of $n$, write $n=2r+1$, use the form $x^2+y^2z-z^{2r}$, and let $C$ be the curve having 
ideal $(x-z^r, y)$. Looking on the patch $Z=1$ of the blow-up gives the surface $X^2 + Y^2z -2z^{2r-2}$ 
and the curve $(X-z^{r-1}, Y)$, so by induction it suffices to prove that, for $n=5$, that is, $r=2$, 
the curve $C$ having ideal $(x-z^2, y)$ gives one of the classes $u_4, u_5$ in $\Picloc p$, where $S$ 
has equation $x^2+y^2z-2z^4$ and $p$ is the origin. On the patch $Z=1$ of the blow-up $\tilde{S}$, 
the equation is $X^2+Y^2z-Z^2$, which has the $\rdA_3$ at the origin, and $\tilde{C}$ has ideal $(X-z, Y)$. 
A further blow-up, again  on the patch $Z=1$, gives equation $X^2+Y^2z-1$ for $\tilde{S}$ and ideal $(X-1, Y)$, 
for $\tilde{C}$. $C$ meets the exceptional locus only at the smooth point with coordinates $(1,0,0)$ relative 
to this patch, so in the full resolution of singularities, $\tilde{C}$ meets one of the two exceptional 
curves arising from the blow-up of the $\rdA_3$, which correspond to $E_4$ and $E_5$ in the Dynkin diagram.

\section{Lines through a point}

We compute the local Picard group $\Picloc p$ of a very general surface 
$S$ containing $r$ general lines passing through $p$, as well as for 
some special configurations of lines. 

\begin{prop}\label{cone}
Let $Z \subset \Ptwo$ be the union of $r$ distinct points and assume 
that the general minimal degree curve $V$ containing $Z$ is smooth. 
Let $Y \subset \Pthree$ be the cone over $Z$. 
Then the very general surface $S$ of 
high degree containing $Y$ is normal and there is a map 
\[
\theta: \Cl \O_{S,p} \to \Pic V / \langle \O_V (1) \rangle
\]
sending the class of the line $L \subset Y$ to the corresponding point in $Z$. 
Moreover, $\theta$ can be identified with the inclusion 
$\Cl \O_{S,p} \hookrightarrow \Cl \widehat \O_{S,p}$ if $\deg V \leq 3$. 
\end{prop}

\begin{proof}
Since we are interested in local properties of $Y$ at $p$, 
consider the affine cone $C(Z) \subset \A^3$ with vertex $p = (0,0,0)$. 
We may assume that $p_1 = (0,0,1)$ so that the line 
$L_1 \subset C(Z)$ has equation $x = y = 0$. 
The ideals of $C(Z)$ and $Z$ agree, so the general surface 
$S$ containing $C(Z)$ has equation $F+G=0$ where $F=0$ defines 
a smooth curve $V \subset \Ptwo$ of degree $d$ and $G$ 
consists of higher degree terms. Blowing up $S$ at $p$, 
the exceptional divisor $E \subset \widetilde S$, 
the projectivized tangent cone for the equation $F+G$, 
appears as the plane curve $V \subset \Ptwo$ and the strict transform 
$\widetilde C(Z)$ consists of $r$ disjoint strict transforms $\widetilde L_i$ 
meeting $E$ transversely at the corresponding points $p_i \in V$. 
For example, on the affine patch $Z=1$ with new variables $x=zX, y=zY$ 
the equation of $E$ is $z^d (F(X,Y,1) + z G(X,Y,z))$ and the strict transform 
of the line $x=y=0$ is $X=Y=0$ corresponding to $(0,0,1)$. 
The self-intersection $E \cap E$ in $\widetilde S$ is given by the line bundle 
$\O_E (-1)$ via the embedding $E \cong V \subset \Ptwo$ because $E$ is the relative 
$\O (1)$ for the proj construction of the blow-up. 

Letting $X = \Spec \widehat \O_{S,p}$, the blow-up $X^\prime \to X$ at $p$ is smooth with 
exceptional divisor $E \cong V$ as above. Now $\Pic X^\prime = \Pic {\widehat X}^\prime$, 
where ${\widehat X}^\prime$ is the formal completion of $X^\prime$ along $E$ \cite[III, 5.1.4]{G}
and $\displaystyle \Pic {\widehat X}^\prime = \lim_{\longleftarrow} \Pic Y_n$, 
where $Y_n$ is the scheme structure on $E$ given by the ideal $\I_E^n$ \cite[II, Exercise 9.6]{AG}. 
The exact sequences 
\[
0 \to \I_E^n / \I_E^{n+1} \to \O_{Y_{n+1}} \to \O_{Y_n} \to 0
\]
yield long exact cohomology sequence fragments
\[
H^1(\I_E^n / \I_E^{n+1}) \to \Pic Y_{n+1} \to \Pic Y_n \to H^2(\I_E^n / \I_E^{n+1})
\]
which show that the maps $\Pic Y_{n+1} \to \Pic Y_n$ are surjective by Grothendieck's vanishing theorem 
\cite[III, Theorem 2.7]{AG}, since the $\I_E^n / \I_E^{n+1}$ are coherent sheaves on $E$. 
Therefore the composite map 
$\displaystyle \Pic X^\prime \cong \Pic {\widehat X}^\prime = \lim_{\longleftarrow} \Pic Y_n \to \Pic Y_1 = \Pic E$
is also surjective. 

The formal Picard group of $S$ at $p$ is 
$\Cl \widehat \O_{S,p} = \Cl X  \cong \Pic (X - \{p\}) \cong \Pic (X^\prime - E)$ and 
$E \cong V$ as above with $E \cap E$ corresponding to $\O_V (-1)$ under this isomorphism, 
so we obtain the commutative diagram 
\begin{equation}\label{CD}
\begin{array}{ccccccccc}
0 & \to & \Z & \stackrel{\cdot E}{\to} & \Pic X^\prime & \to & \Pic (X^\prime - E) \cong \Cl \widehat \O_{S,p} & \to & 0 \\
& & \downarrow & & \downarrow & & \downarrow & & \\
0 & \to & \Z & \stackrel{\cdot E \cap E}{\to} & \Pic E & \to & \Pic E / \langle E \cap E \rangle \cong \Pic V / \langle \O_V (1) \rangle & \to & 0 
\end{array}
\end{equation}
Composing the inclusion $\Cl \O_{S,p} \hookrightarrow \Cl \widehat \O_{S,p}$ with vertical map on the 
right defines $\theta$ and the images of the lines $L \subset Y$ generate $\Cl \O_{S,p}$ as a 
subgroup of $\Cl \widehat \O_{S,p} \cong \Pic (X^\prime - E)$ by Theorem \ref{gens}; 
these classes map to the corresponding points $p \in Z$ under the vertical map on the right, verifying 
the statement about the map $\theta$. 

If $d \leq 3$, then $H^1 (E, \I_E^n / \I_E^{n+1}) \cong H^1 (V, \O_V (n)) = 0$ so the maps 
$\Pic Y_{n+1} \to \Pic Y_n$ are isomorphisms and so are the vertical maps in 
diagram (\ref{CD}), therefore $\theta$ is injective. 
\end{proof}

\begin{cor}\label{rlines} Let $Y$ be the union of $r$ general lines $L_i$ passing through a point $p$. 
Then the very general surface $S$ of high degree containing $Y$ is normal and the local Picard group 
is described as follows:
\begin{enumerate}
\item[(a)] $r \leq 2 \Rightarrow \Cl \O_{S,p}=0$.
\item[(b)] $3 \leq r \leq 5 \Rightarrow L_i \mapsto 1$ under an isomorphism $\Cl \O_{S,p} \cong \Z / 2 \Z$. 
\item[(c)] $r \geq 6 \Rightarrow L_i \mapsto e_i$ under an isomorphism $\Cl \O_{S,p} \cong \Z^{r}$. 
\end{enumerate}
\end{cor}

\begin{proof}
The cases $r=1$ or $r=2$ are clear because $S$ is smooth at $p$.
For part (b), $r$ general points in $\Ptwo$ lie on a smooth conic $V$. 
According to the proposition with $d=2$, there is an injection 
$\Cl \O_{S,p} \hookrightarrow \Pic V / \langle \O_V (1) \rangle \cong \Z / 2 \Z$ 
sending the classes $L_i$ to the corresponding points $p_i \in V$, meaning the 
line bundle $\O_V (p_i)$ of degree $1$.  

Now suppose $r \geq 6$. Letting $d$ be the largest integer satisfying 
$\displaystyle \left(\begin{array}{c}d+1 \\ 2 \end{array}\right) \leq r$, 
$r$ points in general position lie on smooth curves $V$ of degree $d$ and on no curves 
of smaller degree because $r$ points in general position impose independent conditions 
of fixed degree. Since $r \geq 6 \Rightarrow d \geq 3$, the smooth curves $V$ are 
not rational. 

We claim that for sufficiently general $p_i$ and $V$, the $p_i$ generate a free subgroup in 
$\Pic V / \langle \O_V (1) \rangle$. To see this, let 
\[
W = \{(p_1,p_2, \dots p_r,V) \in (\Ptwo)^r \times \mathbb P H^0 (\O_{\Ptwo} (d)): 
p_i \in V, V \;\;{\rm {smooth}}\}
\]
so that each $w \in W$ is associated to a tuple $(p_1^w, p_2^w, \dots p_r^w, V^w)$. 
For $W$ there is the universal curve 
$\mathcal V = \{(x,w) \in \Ptwo \times W:x \in V^w\} \subset \Ptwo \times W$ 
and $r$ effective Cartier divisors 
$\mathcal P_i = \{(x,w) \in \Ptwo \times W: x = p_i^w\} \subset \mathcal V$ 
giving rise to corresponding line bundles 
$\mathcal L_i = \O_{\mathcal V} (\mathcal P_i)$ 
on $\mathcal V$. For fixed tuple $(n_i,m)=(n_1,\dots,n_r,m) \in \Z^{r+1}$, 
consider the subset $U(n_i,m) \subset W$ defined by 
$\O_V (\sum n_i p_i) \otimes \O_V (m) \neq 0$ in $\Pic V$. 
If $\sum n_i + md \neq 0$, then $U(n_i,m) = W$ by reason of degree.  
If $\sum n_i + md=0$, consider the line bundle 
${\mathcal M} = \otimes {\mathcal L_i}^{n_i} \otimes \O_{\mathcal V} (m)$
on $\mathcal V$. 
Since ${\mathcal M}_{V^w}$ is a line bundle of degree zero, 
it is non-trivial in $\Pic V^w$ if and only if $H^0(V^w,{\mathcal  dM}_{V^w}) =0$, 
but this condition is open in $W$ by semicontinuity (${\mathcal M}$ is flat 
over $W$ by constancy of Hilbert polynomial). 
Thus each $U(n_i,m) \subset W$ is an open, dense subset mapping dominantly 
to $(\Ptwo)^r$. 
Taking the countable intersection over all such tuples $(n_i,m)$, we find that the 
general set of $r$ points is contained in a smooth degree $d$ curve $V$ with no 
non-trivial relations in $\Pic V /  \langle \O_V (1) \rangle$. 

Since the $p_i$ have no relations, neither do the lines $L_i$ under the map $\theta$ from 
Proposition \ref{cone}, therefore $\Cl \O_{S,p}$ is a group generated by the $r$ lines 
$L_i$ which surjects onto the free group $\Z^r$, hence the kernel is zero and the $L_i$ 
map to the standard generators $e_i$. 
\end{proof}

\begin{ex}\label{pinwheel}{\em 
In Corollary \ref{rlines} it is essential that the lines be in general position. 

(a) If $Y$ consists of $r$ planar lines through $p$, the general surface $S$ containing $Y$ is 
smooth at $p$ and $\Cl S = \Pic S$ is freely generated by $\O_S (1)$ and the lines $L_{i}$. 

(b) For a more interesting configuration, fix $r \geq 2$ distinct $c_1, \dots c_r \in \C$ 
and consider the curve $Y \subset \Pthree$ with ideal 
\[
I_Y = (xy,yz,\Pi_{j=1}^{r} (x+c_j z)).
\] 
Clearly $Y$ is the union of $r$ lines $L_1, \dots L_r$ with ideals $I_{L_j} = (y,x+c_j z)$ 
and another line $L_0$ with ideal $(x,z)$; $Y$ resembles a pinwheel. 
The general surface $S$ containing $Y$ has local equation 
\[
xy + ayz + b \Pi (x+a_j z)
\]
for units $a,b \in \O_{\Pthree,p}$. 
Setting $X = x + az$ this becomes $X y + b \Pi (X + (c_j - a) z)$. 
Now write $b \Pi (X + (c_j-a) z) = X P + b \Pi (c_j - a) z^r$ and set 
$Y = y + P$ so that the equation for $S$ becomes $X Y - u z^r$, 
where $u = - b \Pi (c_j - a)$ is a unit, exhibiting the equation of an 
$A_{r-1}$ singularity, which has formal Picard group 
$\Cl \widehat \O_{S,p} \cong \mathbb Z / r \mathbb Z$. 

The image of the line $L_0$ is easy to identify, for its ideal is 
$(x, z)=(X, z)$, which corresponds to $1 \in \Z / r \Z$ by Example \ref{image} (a). 
For $1 \leq j \leq r$, the ideal for $L_j$ is
\[
(y, x + c_j z)=(y, X + (c_j - a) z)=(Y - P, X + (c_j - a) z)
\]
Noting that $P \equiv v z^{r-1}$ mod $(X+(c_j-a) z)$ for some scalar $v$ 
(because $P$ is homogeneous in $X$ and $z$ of degree $r-1$), the ideal 
becomes $(X+(c_j-a) z, Y - v z^{r-1})$, which corresponds to $r-1$ in 
$\Z / r \Z$ by Example \ref{image}(c), so $L_j$ corresponds to $-1$.
\em}\end{ex}

\begin{ex}\label{rational}{\em
Fix a smooth conic $V \subset \Ptwo$ and $r \geq 5$ points $p_i \in V$, letting $Y \subset \Pthree$ 
be the cone over $Z = \{p_1, \dots, p_r\}$ with vertex $p$. 
Since the conic $V$ is uniquely determined by $Z$, Proposition \ref{cone} tells us that the 
very general high degree surface $S$ containing $Y$ is normal and that 
$\Cl \O_{S,p} = \Cl \widehat \O_{S,p} \cong \Pic V / \langle \O_V (1) \rangle \cong \Z / 2 \Z$ 
is generated by the class of the lines $L_i \subset Y$. 
 
In fact, $p$ is the {\em only} singularity on $S$ and $\Cl S \cong \Z^{r+1}$ 
is freely generated by $O_S (1)$ and the lines $L_i$ by \cite[Proposition 2.2 and Theorem 1.1]{BN}, 
so Jaffe's exact sequence (see Remark \ref{kadinsky} (a)) simplifies to 
\[
0 \to \Pic S \to \Cl S = \Z^{r+1} \to \Picloc p \cong \Z / 2 \Z \to 0.
\]
Since the lines $L_i$ map to the generator of $\Picloc p$ and $\O_S (1)$ maps to zero, 
we can read off the Picard group of $S$ as 
$\Pic S = \{\O_S (a) + \sum b_i L_i \in \Cl S: 2 | \sum b_i\}$. 
We compute many interesting examples of Picard groups of singular surfaces in \cite{picgps}. 
\em}\end{ex}

Corollary \ref{rlines} is interesting because we are able to compute $\Cl (\O_{S,p})$ and 
$\Pic S$ even though we have not identified $\Cl (\widetilde \O_{S,p})$ or the analytic isomorphism 
class of the singularity $p$. We remedy this in our last result, which addresses Srinivas's question~\cite[Question 3.1]{srinivas} mentioned in the introduction.

\begin{prop}\label{ellipticcone}
Let $V \subset \Pthree$ be the cone over a smooth elliptic curve $C$ with vertex $p$. 
Then for any finitely generated subgroup $H$ of 
$G = \Cl (\widehat \O_{V,p}) \cong \Pic C / \langle \O_C (1) \rangle$, 
there exists a surface $S \subset \Pthree$ having singularity at $q$ analytically 
isomorphic to $p \in V$ with class group isomorphic to $H$ or a subgroup $H' \subset G$ 
such that $|H':H| = 3$.
\end{prop}

\begin{proof}
The subgroup $H$ is generated by the classes in $G$ of a finite number of divisors on $C$. 
First suppose that $H=\cyc{P_1, \dots, P_r}$ is generated by points $P_i \in C$. 
Then $P_1, \dots, P_r$ are cut out on $C$ by smooth curves $D_1$ and $D_2$, whose degrees 
$d_1, d_2$ we can choose as large as we like. Since $V$ has an isolated singularity at $p$, 
$J_F = (F, F_x, F_y, F_z)$ contains some power $\fm^k$ of the maximal ideal $\fm = \fm_{V,p}$; 
choose $D_i$ of degree $\geq k+2$. Let $T_i$ be the defining (homogeneous) polynomial of $D_i$, 
and consider the base locus $Z$ defined in $\P^3$ by the ideal $(F, T_1, T_2)$. 
By construction, its support is the union of the lines $L_j, j=1, \dots, r$ joining $P_j$ to $p$. 
The very general surface $S$ containing $Z$ has class group generated by the $L_i$ by Theorem \ref{gens}; 
furthermore, its equation is of the form $H=F + aT_1 + bT_2$, where $a$ and $b$ are polynomials that 
are units in $\O_{\P^3,p}$. Now, $H, H_x, H_y, H_Z \in J_F$ by construction, since the contribution 
from the $G_i$ are in $\fm^{k+1}$. Also, $H, H_x, H_y, H_z$ generate $J_F$ mod $\fm J_F$, so they 
generate $J_F$. By Mather-Yau, $\O_{S,p}$ is analytically isomorphic to $\O_{V_p}$. As in the 
proof of Theorem \ref{rlines}, we can identify the lines with their corresponding points on $C$, 
so they generate the desired subgroup $H$ of $G$.

For the general case we claim that any generator $F$ of $H$ can be taken to have the 
form $P$ or $P-Q$ for suitable $P, Q \in C$: If $F = (P_1+ ... + P_m) - (Q_1+...+Q_n)$, where 
the $P_i, Q_j$ can repeat and $m>1$, then there is a unique $R \in C$ such that 
$P_1+P_2+R \in |H^0 (\O_C(1))|$, so $F$ is congruent to $(P_3+...+P_n) - (Q_1+...+Q_n+R)$, 
and this involves fewer than $m+n$ points. A similar calculation holds when $n>1$; by induction 
we get it down to either $P, -Q$ (which generates the same subgroup as $Q$), or $P-Q$, as claimed. 
Furthermore, any generator $P-Q$ can be written as $P'-Q_0$ for any desired $Q_0$; this follows 
from the group law on $C$.

Now write $H=\cyc{P_1, P_2, \dots, P_r, P'_1-Q_0, \dots P'_s - Q_0}$. 
If $r \geq 1$, take $Q_0$ to be $P_1$ so that $H=\cyc{P_1, P'_1, \dots, P'_s}$ and we are back 
in the first case. Otherwise, take $Q_0$ to be an inflection point, so that $3Q_0 = 0$ in $G$; 
now $H \subset H'=\cyc{P_1, P_2, \dots, P_r, P'_1, \dots P'_s, Q_0}$ is a subgroup of index $3$, 
and as above $H'$ can be realized as the class group of a singularity analytically isomorphic 
to $\Spec\O_{V,p}$.
\end{proof}

\begin{rmk}{\em
It is interesting to note that while the curves $F, D_1, D_2$ generate the ideal 
{\em sheaf} of the union of the points $P_i$, the fact that they do not generate the 
{\em total} ideal creates an embedded point at the origin in the intersection of their cones; 
this is what guarantees the correct analytic isomorphism class of the very general surface.}
\end{rmk}

\begin{rmk}{\em
The method employed cannot be used to produce surfaces having the desired analytic isomorphism class 
and subgroups generated purely by elements of degree $0$; however, this is far from a proof that no such surface exists. We know of no example of a finitely generated subgroup of the class group of, say, a complete normal hypersurface singularity that provably cannot occur as the class group of an analytically 
isomorphic surface singularity. Srinivas in ~\cite[Example 3.9]{srinivas} makes a similar remark {\em vis-\`a-vis} 
the surface singularity defined by $x^2+y^3+z^7$ -- which is a UFD whose completion has class group $\C$ -- 
that ``presumably the finitely generated subgroup [{\em i.e.}, the class group of an analytically isomorphic 
singular point] \dots can be of arbitrary rank \dots" To be sure, our example is easier to work with, since 
we have at our disposal a specific local ring with a ready supply of algebraic cycles to use as divisors. With the evidence at hand, we are led to pose the following.
}\end{rmk}

\begin{ques} For a given complete Gorenstein ring $A$ essentially of finite type over $\C$, 
does {\em every} finitely generated subgroup of $\Cl A$ arise as the image of 
$\Cl B \hookrightarrow \Cl A$ for some local ring of essentially finite type over $\C$ such that $\widehat B = A$?  \end{ques}

\end{document}